\documentclass[11pt]{amsart}
\usepackage{amsmath,amsthm,amssymb,amscd}

\allowdisplaybreaks

\theoremstyle{plain}

\newtheorem{thm}{Theorem}

\newtheorem{lemma}[thm]{Lemma}

\newtheorem{lem}[thm]{Lemma}

\theoremstyle{definition}

\newtheorem{ex}[thm]{Example}

\newcommand\cE{{\mathcal{E}}}
\newcommand\cF{{\mathcal{F}}}

\newcommand\cL{{\mathcal{L}}}

\newcommand\cR{{\mathcal{R}}}

\newcommand\cV{{\mathcal{V}}}
\newcommand\cW{{\mathcal{W}}}

\newcommand\RR{\mathbb{R}}

\newcommand\ZZ{\mathbb{Z}}

\newcommand\CC{\mathbb{C}}

\newcommand\NN{\mathbb{N}}

\newcommand\Ker{\mathop{\rm Ker}\nolimits}

\newcommand\End{\operatorname{End}}
\newcommand{\Tr}{\mbox{\rm Tr}}

\begin{document}
\title[Transverse Dirac operators on Riemannian foliations]{Vanishing theorem
for transverse Dirac operators on Riemannian foliations}
\author{Yuri A. Kordyukov}
\address{Institute of Mathematics, Russian Academy of Sciences, 112
Chernyshevsky street, 450077 Ufa, Russia}
\email{yurikor@matem.anrb.ru}\thanks{Supported by the Russian
Foundation of Basic Research (grant no. 06-01-00208)}




\begin{abstract}
We obtain a vanishing theorem for the half-kernel of a transverse
${\rm Spin}\sp c$ Dirac operator on a compact manifold endowed with
a transversely almost complex Riemannian foliation twisted by a
sufficiently large power of a line bundle, whose curvature vanishes
along the leaves and is transversely non-degenerate at any point of
the ambient manifold.
\end{abstract}
\maketitle \hyphenation{trans-ver-sal-ly}


\bibliographystyle{plain}

\section*{Introduction}
Let $X$ be a compact manifold of dimension $2n$ equipped with an
almost complex structure $J : TX \to TX$, $\cE$ a Hermitian vector
bundle on $X$, and $g_X$ a Riemannian metric on $X$. Assume that the
almost complex structure $J$ is compatible with $g_X$. Consider a
Hermitian line bundle $\cL$ over $X$ endowed with a Hermitian
connection $\nabla^\cL$ such that its curvature
$R^{\cL}=(\nabla^\cL)^2$ is non-degenerate. Thus, $\omega =
\frac{i}{2\pi}R^{\cL}$ is a symplectic form on $X$. One can
construct canonically a ${\rm Spin}\sp c$ Dirac operator $D_k$
acting on
\[
\Omega^{0,*}(X,\cE\otimes \cL^k)=\oplus_{q=0}^n
\Omega^{0,q}(X,\cE\otimes \cL^k),
\]
the direct sum of spaces of $(0,q)$-forms with values in $\cE\otimes
\cL^k$.

Under the assumption that $J$ is compatible with $\omega$, Borthwick
and Uribe \cite{B-Uribe96} proved that, for sufficiently large $k$,
\[
\Ker D^-_k =0,
\]
where $D^-_k$ denotes the restriction of $D_k$ to
$\Omega^{0,odd}(X,\cE\otimes \cL^k)$. This result generalizes the
famous Kodaira vanishing theorem for the cohomology of the sheaf of
sections of a holomorphic vector bundle twisted by a large power of
a positive line bundle. It has interesting applications in geometric
quantization (see \cite{B-Uribe96} and references therein).

In \cite{ma-ma2002}, Ma and Marinescu gave a proof of the
Borthwick-Uribe result, which uses only the Lichnerowicz formula for
the ${\rm Spin}\sp c$ Dirac operator. They also show that, if we put
\[
m=\inf_{u\in T^{(1,0)}_xX, x\in X}
\frac{R^\cL_x(u,\bar{u})}{|u|^2}>0,
\]
then there exists $C>0$ such that, for $k\in \NN$, the spectrum of
$D^2_k$ is contained in the set $\{0\}\cup (2km-C,+\infty)$.

Let us mention that Braverman \cite{Brav,Brav99} generalized the
Borthwick-Uribe vanishing theorem to the case when the almost
complex structure $J$ is compatible with $g_X$, the curvature
$R^{\cL}$ is non-degenerate and $J$-invariant, but not necessarily
compatible with $J$. In \cite{ma-ma2006}, Ma and Marinescu gave a
proof of this result by the methods of \cite{ma-ma2002}.

Our main purpose is to obtain an analogue of the Borthwick-Uribe
vanishing theorem for a transverse ${\rm Spin}\sp c$ Dirac operator
on a compact manifold endowed with a transversely almost complex
Riemannian foliation. Our considerations are based on the approach
of Ma-Marinescu \cite{ma-ma2002}. So we also state a Lichnerowicz
type formula for a transverse Dirac operator on a compact foliated
manifold $(M,\cF)$, which, as we strongly believe, will be of
independent interest.

The transverse Dirac operators for Riemannian foliations were
introduced in \cite{G-K91}. This paper mainly deals with the
transverse Dirac operators acting on basic sections (see also
\cite{G-K91a,G-K93,Jung01,Jung06,Jung07} and references therein).
The index theory of transverse Dirac operators was studied in
\cite{DGKY}. Finally, spectral triples defined by transverse
Dirac-type operators on Riemannian foliations and related
noncommutative geometry were considered in
\cite{noncom,mpag,matrix-egorov}.

The paper is organized as follows. In Section~\ref{s:results}, we
introduce transverse Dirac operators and formulate our main results,
the vanishing theorem for a transverse ${\rm Spin}\sp c$ Dirac
operator on a compact manifold endowed with a transversely almost
complex Riemannian foliation, Theorem~\ref{t:vanish}, and the
Lichnerowicz formula for transverse Dirac operators,
Theorem~\ref{t:Lich}. The proof of the vanishing theorem is given in
Section~\ref{s:van}. Sections~\ref{s:Lich} and~\ref{s:Bochner}
contain the proof of the Lichnerowicz formula and related results.

The author is grateful to M. Braverman and G. Marinescu for useful
remarks.

\section{Preliminaries and main results}\label{s:results}
\subsection{Transverse Dirac operators}
Let $M$ be a compact manifold equipped with a Riemannian foliation
$\cF$ of even codimension $q$ and $g_M$ a bundle-like metric on $M$.
Let $T^H_xM=T_x{\cF}^{\bot}$. $T^HM$ is a smooth vector subbundle of
$TM$ such that
\begin{equation}\label{e:decomp}
TM=T^HM\oplus T\cF.
\end{equation}
There is a natural isomorphism $T^HM\cong Q=TM/T\cF$. Denote by
$P_H$ (resp. $P_F$) the orthogonal projection operator of $TM$ on
$T^HM$ (resp. $T\cF$) associated with the decomposition
(\ref{e:decomp}).

The Riemannian metric $g_M$ gives rise to a metric connection
$\nabla$ in $T^HM$ (called the transverse Levi-Civita connection),
which is defined as follows. Denote by $\nabla^L$ the Levi-Civita
connection defined by $g_M$. Then we have
\begin{equation}\label{e:adapt}
\begin{aligned}
\nabla_XN &=P_H[X,N],\quad X\in C^\infty(M,T\cF),\quad N\in
C^\infty(M,T^HM)\\ \nabla_XN&=P_H\nabla^L_XN,\quad X\in
C^\infty(M,T^HM),\quad N\in C^\infty(M,T^HM).
\end{aligned}
\end{equation}
It turns out that $\nabla$ depends only on the transverse part of
the metric $g_M$ and preserves the inner product of $T^HM$.

For any $x\in M$, denote by $Cl(Q_x)$ the Clifford algebra of $Q_x$.
Recall that, relative to an orthonormal basis
$\{f_1,f_2,\ldots,f_q\}$ of $Q_x$, $Cl(Q_x)$ is the complex algebra
generated by $1$ and $f_1,f_2,\ldots,f_q$ with relations
\[
f_\alpha f_\beta+f_\beta f_\alpha=-2\delta_{\alpha\beta}, \quad
\alpha, \beta=1,2,\ldots,q.
\]

The transverse Clifford bundle $Cl(Q)$ is the $\ZZ_2$-graded vector
bundle over $M$ whose fiber at $x\in M$ is $Cl(Q_x)$. This bundle is
associated to the principal $SO(q)$-bundle $O(Q)$ of oriented
orthonormal frames in $Q$, $Cl(Q)=O(Q)\times_{O(q)}Cl(\RR^q)$.
Therefore, the transverse Levi-Civita connection $\nabla$ induces a
natural leafwise flat connection $\nabla^{Cl(Q)}$ on $Cl(Q)$ which
is compatible with the multiplication and preserves the
$\ZZ_2$-grading on $Cl(Q)$. If $\{f_1, f_2, \ldots , f_q\}$ is a
local orthonormal frame in $T^HM$, and $\omega^\gamma_{\alpha\beta}$
is the coefficients of the connection $\nabla$:
$\nabla_{f_\alpha}f_\beta
=\sum_\gamma\omega^\gamma_{\alpha\beta}f_\gamma$, then
\begin{equation}\label{e:spin}
\nabla^{Cl(Q)}_{f_\alpha} = f_\alpha+\frac{1}{4}\sum_{\gamma=1}^q
\omega^\gamma_{\alpha\beta}c(f_\beta)c(f_\gamma),
\end{equation}
where $c(a)$ denotes the action of an element $a\in Cl(Q)$ on
$C^\infty(M,Cl(Q))$ by pointwise left multiplication.

A transverse Clifford module is a complex vector bundle $\cE$ on $M$
endowed with an action of the bundle $Cl(Q)$. We will denote the
action of $a\in C^\infty(M,Cl(Q))$ on $s\in C^\infty(M,\cE)$ as
$c(a)s\in C^\infty(M,\cE)$.

A transverse Clifford module $\cE$ is called self-adjoint if it
endowed with a leafwise flat Hermitian metric such that the operator
$c(f) : \cE_x\to \cE_x$ is skew-adjoint for any $x\in M$ and $f\in
Q_x$.

Any transverse Clifford module $\cE$ carries a natural
$\ZZ_2$-grading $\cE=\cE_+\oplus\cE_-$ (see, for instance,
\cite{BGV92}).

A connection $\nabla^\cE$ on a transverse Clifford module $\cE$ is
called a Clifford connection if it is compatible with the Clifford
action, that is, for any $f\in C^\infty(M,T^HM)$ and $a\in
C^\infty(M,Cl(Q))$,
\[
[\nabla^{\cE}_f, c(a)]=c(\nabla^{Cl(Q)}_fa).
\]

\begin{ex} Assume that $\cF$ is transversely oriented and the
normal bundle $Q$ is spin. Thus the $SO(q)$ bundle $O(Q)$ of
oriented orthonormal frames in $Q$ can be lifted to a $Spin(q)$
bundle $O'(Q)$ so that the projection $O'(Q)\to O(Q)$ induces the
covering projection $Spin(q)\to SO(q)$ on each fiber.

Let $F(Q), F_+(Q), F_-(Q)$ be the bundles of spinors
\[
F(Q)=O'(Q)\times_{Spin(q)}S, \quad
F_\pm(Q)=O'(Q)\times_{Spin(q)}S_\pm .
\]
Since $\dim Q=q$ is even $\End F(Q)$ is as a bundle of algebras over
$M$ isomorphic to the Clifford bundle $Cl(Q)$. So $F(Q)$ is a
self-adjoint transverse Clifford module. The transverse Levi-Civita
connection $\nabla$ lifts to a leafwise flat Clifford connection
$\nabla^{F(Q)}$ on $F(Q)$.

More generally, one can take a Hermitian vector bundle $\cW$
equipped with a leafwise flat Hermitian connection $\nabla^\cW$.
Then $F(Q)\otimes \cW$ is a transverse Clifford module: the action
of $a\in C^\infty(M,Cl(Q))$ on $C^\infty(M, F(Q)\otimes \cW)$ is
given by $c(a)\otimes 1$ ($c(a)$ denotes the action of $a$ on
$C^\infty(M, F(Q))$). The product connection $\nabla^{F(Q)\otimes
\cW}=\nabla^{F(Q)}\otimes 1 + 1\otimes \nabla^{\cW}$ on $F(Q)\otimes
\cW$ is a Clifford connection.
\end{ex}

\begin{ex}
Another example of a self-adjoint transverse Clifford module
associated with a transverse almost complex structure on $(M,\cF)$,
a transverse Clifford module $\Lambda^{0,*}$, is described in
Section~\ref{s:van1}.
\end{ex}

Let $\cE$ be a self-adjoint transverse Clifford module equipped with
a leafwise flat Clifford connection $\nabla^\cE$. We will identify
the bundle $Q$ and $Q^*$ by means of the metric $g_M$ and define the
operator $D^\prime_\cE$ acting on the sections of $\cE$ as the
composition
\[
C^\infty(M,\cE)\stackrel{\nabla^{\cE}}{\longrightarrow}
C^\infty(M,Q^*\otimes \cE) = C^\infty(M,Q\otimes \cE)
\stackrel{c}{\longrightarrow} C^\infty(M, \cE).
\]
This operator is odd with respect to the natural $\ZZ_2$-grading on
$\cE$. If $f_1,\ldots,f_q$ is a local orthonormal frame for $T^HM$,
then
\[
D^\prime_\cE=\sum_{\alpha=1}^qc(f_\alpha)\nabla^{\cE}_{f_\alpha}.
\]

Let $\tau \in C^\infty(M,T^HM)$ be the mean curvature vector field
of $\cF$. If $e_1,e_2,\ldots,e_p$ is a local orthonormal frame in
$T\cF$, then
\[
\tau=\sum_{i=1}^pP_H(\nabla^L_{e_i}e_i).
\]
The transverse Dirac operator $D_\cE$ is defined as
\[
D_\cE=D^\prime_\cE-\frac12 c(\tau)= \sum_{\alpha=1}^q
c(f_\alpha)\left(\nabla^{\cE}_{f_\alpha}-\frac12 g_M(\tau, f_\alpha)
\right).
\]

Denote by $(\cdot,\cdot)_x$ the inner product in the fiber $\cE_x$
over $x\in M$. Then the inner product in $L^2(M, \cE)$ is given by
the formula
\[
(s_1, s_2)=\int_M (s_1(x),s_2(x))_x\omega_M, \quad s_1, s_2\in
L^2(M, \cE),
\]
where $\omega_M=\sqrt{\det g}\,dx$ denotes the Riemannian volume
form on $M$. As shown in \cite{matrix-egorov}, the transverse Dirac
operator $D_\cE$ is formally self-adjoint in $L^2(M, \cE)$.

\subsection{The vanishing theorem}\label{s:van1}
As above, let $M$ be a compact manifold equipped with a Riemannian
foliation $\cF$ of even codimension $q$, $g_M$ a bundle-like metric
on $M$. Consider a Hermitian line bundle $\cL$ equipped with a
leafwise flat Hermitian connection $\nabla^\cL$.

The curvature of $\nabla^\cL$ is an imaginary valued $2$-form
$R^\cL=(\nabla^\cL)^2$ on $M$. Since $\nabla^\cL$ is leafwise flat,
$R^\cL$ vanishes on $T\cF$, and, therefore, defines a $2$-form
$R^\cL$ on $Q$. If this form is non-degenerate, then it is a
symplectic form on $Q$.

Let $J :Q\to Q$ be an almost complex structure, which is compatible
with $g_Q$ and $R^\cL$. The almost complex structure $J$ defines
canonically an orientation of $Q$ and induces a splitting $Q\otimes
\CC=Q^{(1,0)}\oplus Q^{(0,1)}$, where $Q^{(1,0)}$ and $Q^{(0,1)}$
are the eigenbundles of $J$ corresponding to the eigenvalues $i$ and
$-i$ respectively. We also have the corresponding decomposition of
the complexified conormal bundle $Q^*\otimes \CC=Q^{(1,0)*}\oplus
Q^{(0,1)*}$ and the decomposition of the exterior algebra bundles
$\Lambda (Q^*\otimes \CC)=\oplus_{p,q}\Lambda^{p,q}(Q^*\otimes
\CC)$, where $\Lambda^{p,q}(Q^*\otimes \CC) = \Lambda^p Q^{(1,0)*}
\otimes \Lambda^qQ^{(0,1)*}$. The transverse Levi-Civita connection
$\nabla$ can be written as
\[
\nabla=\nabla^{(1,0)}+ \nabla^{(0,1)}+A,
\]
where $\nabla^{(1,0)}$ and $\nabla^{(0,1)}$ are the canonical
Hermitian connections on $Q^{(1,0)}$ and $Q^{(0,1)}$ respectively
and $A\in C^\infty(T^*M\otimes \End (Q))$, which satisfies $JA=-AJ$.

Consider a self-adjoint transverse Clifford module
\[
\Lambda^{0,*}=\Lambda^{even}Q^{(0,1)*}\oplus
\Lambda^{odd}Q^{(0,1)*}.
\]
The action of any $f\in Q$ with decomposition $f=f_{1,0}+f_{0,1}\in
Q^{(1,0)}\oplus Q^{(0,1)}$ on $\Lambda^{0,*}$ is defined as
\[
c(f)=\sqrt{2}(\varepsilon_{f^*_{1,0}}-i_{f_{0,1}}),
\]
where $\varepsilon_{f^*_{1,0}}$ denotes the exterior product by the
covector ${f^*_{1,0}}\in Q_x^*$ dual to ${f_{1,0}}$, $i_{f_{0,1}}$
the interior product by ${f_{0,1}}$. This module has a natural
leafwise flat Clifford connection $\nabla^{\Lambda^{0,*}}$. The
associated transverse Dirac operator $D_{\Lambda^{0,*}}$ can be
called the transverse ${\rm Spin}\sp c$ Dirac operator.

One can also consider a Hermitian vector bundle $\cW$ equipped with
a leafwise flat Hermitian connection $\nabla^\cW$. Then one get the
twisted transverse Clifford module $\cE=\Lambda^{0,*}\otimes\cW$
equipped with a product leafwise flat Hermitian connection
$\nabla^\cE$ and the associated transverse ${\rm Spin}\sp c$ Dirac
operator $D_{\Lambda^{0,*}\otimes \cW}$.

Consider the transverse ${\rm Spin}\sp c$ Dirac operator
\[
D_k=D_{\Lambda^{0,*} \otimes \cW \otimes \cL^k} :
C^\infty(M,\Lambda^{0,*} \otimes \cW\otimes \cL^k) \to
C^\infty(M,\Lambda^{0,*} \otimes \cW \otimes \cL^k).
\]
Let $D^-_k$ denote the restriction of $D_k$ to the space
$C^\infty(M,\Lambda^{odd}Q^{(0,1)*}\otimes \cW \otimes \cL^k)$. Put
\[
m=\inf_{u\in Q^{(1,0)}_x, x\in M}
\frac{R^\cL_x(u,\bar{u})}{|u|^2}>0.
\]

\begin{thm}\label{t:vanish}
There exists $C>0$ such that for $k\in \NN$, the spectrum of $D^2_k$
is contained in the set $\{0\}\cup (2km-C,+\infty)$. For
sufficiently large $k$
\[
\Ker D^-_k =0.
\]
\end{thm}

The proof of this theorem will be given in Section~\ref{s:van}.

\subsection{The Lichnerowicz formula}
In this Section, we will formulate the Lichnerowicz formula for a
transverse Dirac operator, which will play a crucial role in the
proof of Theorem~\ref{t:vanish}.

Denote by $\cR$ the integrability tensor (or curvature) of $T^HM$.
It is an element of $C^\infty(M, \Lambda^2T^HM^* \otimes T\cF)$
given by
\[
\cR_x(f_1,f_2)=-P_F[\tilde{f}_1,\tilde{f}_2](x),\quad x\in M, \quad
f_1, f_2\in T^H_xM,
\]
where, for any $f\in T^H_xM$, $\tilde{f}\in C^\infty(M,T^HM)$
denotes any vector field, which coincides with $f$ at $x$.

Since the Levi-Civita connection $\nabla^L$ is torsion-free, for any
$f_1, f_2\in C^\infty(M,T^HM)$, we have
\begin{equation}
\label{e:R} \nabla_{f_1}f_2-\nabla_{f_2}f_1-[f_1,f_2]=\cR(f_1,f_2).
\end{equation}

Let $R$ be the curvature of the transverse Levi-Civita connection
$\nabla$. By definition, $R$ is a section of $\Lambda^2T^*M \otimes
\End (T^HM)$ given by the formula
\[
R(X,Y)=\nabla_{X}\nabla_{Y}-\nabla_{Y}\nabla_{X}-\nabla_{[X,Y]},
\quad X, Y\in C^\infty(M,TM).
\]
If $(B,g_B)$ is a local model for the foliation and $R^B$ is the
curvature of $g_B$, then, for any $f_1, f_2, f_3 \in TB$ with the
corresponding horizontal lifts $f^H_1, f^H_2, f^H_3\in T^HM$, we
have
\[
R(f^H_1,f^H_2)f^H_3=[R^B(f_1,f_2)f_3]^H+P_H([\cR(f^H_1,f^H_2),
f^H_3]).
\]

Denote by $R^{\cE}$ the curvature of the Clifford connection
$\nabla^{\cE}$. By definition, $R^{\cE}$ is a section of
$\Lambda^2T^*_HM \otimes \End (\cE)$ given by the formula
\[
R^{\cE}(f_1, f_2)=\nabla^{\cE}_{f_1}\nabla^{\cE}_{f_2}
-\nabla^{\cE}_{f_2}\nabla^{\cE}_{f_1}-\nabla^{\cE}_{[f_1,f_2]}.
\]
It can be written as
\[
R^{\cE}=c(R)+R^{\cE/S},
\]
where $c(R)\in C^\infty(M,\Lambda^2T^*_HM \otimes Cl(Q))$ is
determined by the curvature $R$ of $\nabla$: If $\{f_1, f_2, \ldots
, f_q\}$ is a local orthonormal frame in $T^HM$, then
\[
c(R)(f_1, f_2)= \frac{1}{4}\sum_{\alpha,\beta} (R(f_1, f_2)f_\alpha,
f_\beta)c(f_\alpha)c(f_\beta),
\]
and $R^{\cE/S}\in C^\infty(M,\Lambda^2T^*_HM \otimes \End_{Cl(Q)}
(\cE))$ is the twisting curvature of $\cE$.

Denote by $(\nabla^{\cE}_X)^*$ the formal adjoint of the operator
$\nabla^{\cE}_X$ with $X\in C^\infty(M,T^HM)$ in $L^2(M,\cE)$.
Observe the following formula:
\begin{equation}\label{e:nabla}
(\nabla^{\cE}_X)^*=-\nabla^{\cE}_X - \operatorname{div} X.
\end{equation}
where $\operatorname{div} X\in C^\infty(M)$ denotes the divergence
of $X$. If $e_1,e_2,\ldots,e_p$ is a local orthonormal frame in
$T\cF$ and $f_1,\ldots, f_q$ is a local orthonormal basis of $T^HM$,
then
\[
\operatorname{div} X= \sum_{k=1}^p g_M(e_k,
\nabla_{e_k}X)+\sum_{\beta=1}^q g_M(f_\beta, \nabla_{f_\beta}X).
\]
In particular, it is easy to see that
\[
\operatorname{div} f_\alpha= - g_M(\tau+
\sum_{\beta=1}^q\nabla_{f_\beta}f_\beta, f_\alpha).
\]
Let $f_1,\ldots, f_q$ be a local orthonormal basis of $T^HM$. Define
the transverse scalar curvature $K$ as
\[
K=\sum_{\alpha, \beta} g(R(f_\alpha, f_\beta)f_\alpha, f_\beta).
\]

\begin{thm}\label{t:Lich}
The following formula holds:
\begin{multline*}
(D_\cE)^2 = \sum_{\alpha=1}^q (\nabla^{\cE}_{f_\alpha})^*
\nabla^{\cE}_{f_\alpha} - \frac12 \sum_{\alpha=1}^q
c(f_\alpha)c(\nabla_{f_\alpha}\tau) -\frac14\|\tau\|^2
\\  + \frac{K}{4} +\frac{1}{2}\sum_{\alpha,\beta}
c(f_\alpha)c(f_\beta)[R^{\cE/S}(f_\alpha,f_\beta)-\nabla_{\cR(f_\alpha,f_\beta)}],
\end{multline*}
where $f_1,\ldots, f_q$ is a local orthonormal basis of $T^HM$.
\end{thm}

The proof of this theorem and related results will be given in
Section~\ref{s:Lich}.

\section{Proof of the vanishing theorem}\label{s:van}
The purpose of this Section is to give the proof of
Theorem~\ref{t:vanish}. This proof will make an essential use of
Theorem~\ref{t:Lich}, whose proof will be given later, in
Section~\ref{s:Lich}. First, we introduce some notation.

For any $x\in M$, define the skew-symmetric linear map $K_x : Q_x\to
Q_x$ by the formula
\[
iR^\cL(v,w)=g_Q(v,K_xw), \quad v,w\in Q_x.
\]
The eigenvalues of $K_x$ are purely imaginary: $\pm i\mu_j(x),
j=1,2,\ldots,l$ with $\mu_j(x)>0$. Define
\[
\lambda(x)=\Tr^+ K_x=\mu_1(x)+\ldots+\mu_l(x), \quad m(x)=\min_j
\mu_j(x).
\]
Observe that
\[
m=\min_{x\in M}m(x).
\]
Denote
\[
c(\cR)=\frac{1}{2}\sum_{\alpha,\beta}
c(f_\alpha)c(f_\beta)\nabla_{\cR(f_\alpha,f_\beta)}
\]
and
\[
c(R^\cL)=\frac{1}{2}\sum_{\alpha,\beta}
c(f_\alpha)c(f_\beta)R^{\cL}(f_\alpha,f_\beta).
\]

We start the proof of Theorem~\ref{t:vanish} with the following
lemma, which provides a lower estimate of the transverse metric
Laplacian (cf. \cite[Corollary 2.4]{ma-ma2002} and reference
therein).

\begin{lem}\label{l:1}
Let $\cV$ be a Hermitian vector bundle over $M$ equipped with a
leafwise flat unitary connection $\nabla^\cV$, and $\cL$ a Hermitian
line bundle equipped with a leafwise flat Hermitian connection
$\nabla^\cL$. There exists $C>0$ such that for any $k\in \NN$ the
transverse metric Laplacian
\[
\Delta^{\cV \otimes \cL^k}=\sum_{\alpha=1}^q (\nabla^{\cV \otimes
\cL^k}_{f_\alpha})^* \nabla^{\cV \otimes \cL^k}_{f_\alpha}
\]
satisfies
\[
((\Delta^{\cV \otimes \cL^k}-c(\cR))u,u)\geq k(\lambda u,u)-C\|u\|^2
\]
for any $u\in C^\infty(M,\cV \otimes \cL^k)$.
\end{lem}

\begin{proof}
Consider the twisted transverse Clifford module
$\Lambda^{0,*}\otimes \cV \otimes \cL^k$ and the associated twisted
transverse ${\rm Spin}\sp c$ Dirac operator $D_{\Lambda^{0,*}\otimes
\cV \otimes \cL^k}$. By Theorem~\ref{t:Lich}, we have
\begin{multline*}
D_{\Lambda^{0,*}\otimes \cV \otimes \cL^k}^2 =
\Delta^{\Lambda^{0,*}\otimes \cV \otimes \cL^k} - \frac12
\sum_{\alpha=1}^q (c(f_\alpha)c(\nabla_{f_\alpha}\tau)\otimes 1)
-\frac14\|\tau\|^2 + \frac{K}{4} \\
+c(R^{\cV})-c(\cR) +k c(R^\cL),
\end{multline*}
where $f_1,\ldots, f_q$ is a local orthonormal basis of $T^HM$. From
(\ref{e:spin}), we have
\[
\|\nabla^{\Lambda^{0,*}\otimes\cV \otimes \cL^k}u\|^2=\|\nabla^{\cV
\otimes \cL^k}u\|^2+ \frac{1}{16}\| \sum_\gamma
\omega^\gamma_{\alpha\beta}c(f_\beta)c(f_\gamma)u\|^2.
\]
It can be shown (see, for instance, \cite[Lemma 7.10]{Brav}) that,
for any $u\in (\cV \otimes \cL^k)_x \subset (\Lambda^{0,*}\otimes
\cV \otimes \cL^k)_x$
\begin{equation}\label{e:RL1}
c(R^\cL)u = -\lambda u.
\end{equation}
Using (\ref{e:RL1}), we get, for any $u\in C^\infty(M,\cV \otimes
\cL^k)\subset C^\infty(M,\Lambda^{0,*}\otimes \cV \otimes \cL^k)$,
\[
0 \leq \|D_{\Lambda^{0,*}\otimes \cV \otimes \cL^k}u\|^2 \leq
\|\nabla^{\cV \otimes \cL^k}u\|^2 + C\|u\|^2 - (c(\cR)u,u) -k
(\lambda u,u),
\]
that completes the proof.
\end{proof}

By Theorem~\ref{t:Lich}, we have
\begin{multline*}
D_k^2 = \Delta^{\Lambda^{0,*} \otimes \cW \otimes \cL^k} - \frac12
\sum_{\alpha=1}^q c(f_\alpha)c(\nabla_{f_\alpha}\tau)
-\frac14\|\tau\|^2 + \frac{K}{4} \\
+c(R^{\cW})-c(\cR) +k c(R^\cL),
\end{multline*}
where $f_1,\ldots, f_q$ is a local orthonormal basis of $T^HM$.
Therefore, for any $u\in C^\infty(M,\Lambda^{0,*} \otimes \cW
\otimes \cL^k)$, we have
\[
\|D_ku\|^2\geq (\Delta^{\Lambda^{0,*} \otimes \cW \otimes
\cL^k}u,u)- (c(\cR)u,u)+k (c(R^\cL)u,u)-C\|u\|^2.
\]
By Lemma~\ref{l:1}, it follows that
\[
((\Delta^{\Lambda^{0,*} \otimes \cW \otimes \cL^k}-c(\cR))u,u)\geq
k(\lambda
 u,u)-C\|u\|^2.
\]
So we see that
\[
\|D_ku\|^2\geq k(\lambda u,u) +k (c(R^\cL)u,u)-C\|u\|^2.
\]
Finally, by \cite[Proposition 7.5]{Brav}, we have
\[
(c(R^\cL)u,u)_x\geq -(\lambda(x)-2m(x))\|u\|_x^2, \quad u\in
(\Lambda^{odd}Q^{(0,1)*}\otimes \cW \otimes \cL^k)_x.
\]
Therefore, for $u\in C^\infty(M,\Lambda^{odd}Q^{(0,1)*}\otimes \cW
\otimes \cL^k)$, we get
\[
\|D_ku\|^2\geq 2k(m u,u)-C\|u\|^2,
\]
that immediately completes the proof of Theorem~\ref{t:vanish}.

\section{Proof of the Lichnerowicz formula}\label{s:Lich}
In this Section, we derive the Lichnerowicz formula for a transverse
Dirac operator given in Theorem~\ref{t:Lich}. We start with a
computation of $(D^\prime_\cE)^2$:
\begin{align*}
(D^\prime_\cE)^2 =
&\frac{1}{2}[\left(\sum_{\alpha=1}^qc(f_\alpha)\nabla^{\cE}_{f_\alpha}\right)
\left(\sum_{\beta=1}^qc(f_\beta)\nabla^{\cE}_{f_\beta}\right)\\ &
+\left(\sum_{\beta=1}^qc(f_\beta)\nabla^{\cE}_{f_\beta}\right)
\left(\sum_{\alpha=1}^qc(f_\alpha)\nabla^{\cE}_{f_\alpha}\right)]\\
= & \frac{1}{2}\sum_{\alpha,\beta} (c(f_\alpha)c(f_\beta) +
c(f_\beta)c(f_\alpha)) \nabla^{\cE}_{f_\alpha}\nabla^{\cE}_{f_\beta}
\\ & + \frac{1}{2}\sum_{\alpha,\beta} c(f_\beta)c(f_\alpha) (\nabla^{\cE}_{f_\beta}
\nabla^{\cE}_{f_\alpha} -
\nabla^{\cE}_{f_\alpha}\nabla^{\cE}_{f_\beta}) \\ &+
\frac{1}{2}\sum_{\alpha,\beta}[c(f_\alpha)c(\nabla_{f_\alpha}
f_\beta)\nabla^{\cE}_{f_\beta}+c(f_\beta)c(\nabla_{f_\beta}
f_\alpha)\nabla^{\cE}_{f_\alpha}].
\end{align*}
For the first term, we get
\[
\frac{1}{2}\sum_{\alpha,\beta} (c(f_\alpha)c(f_\beta) +
c(f_\beta)c(f_\alpha)) \nabla^{\cE}_{f_\alpha}\nabla^{\cE}_{f_\beta}
=-\sum_{\alpha} (\nabla^{\cE}_{f_\alpha})^2.
\]
For the second term, we get
\begin{multline*}
\frac{1}{2}\sum_{\alpha,\beta} c(f_\beta)c(f_\alpha)
(\nabla^{\cE}_{f_\beta}\nabla^{\cE}_{f_\alpha} - \nabla^{
\cE}_{f_\alpha}\nabla^{\cE}_{f_\beta})\\ =
\frac{1}{2}\sum_{\alpha,\beta} c(f_\beta)c(f_\alpha)
R^{\cE}(f_\beta, f_\alpha)+ \frac{1}{2}\sum_{\alpha,\beta}
c(f_\beta)c(f_\alpha) \nabla^{\cE}_{[f_\beta, f_\alpha]}.
\end{multline*}
Let $\nabla_{f_\alpha}f_\beta=\sum_\gamma
a^\gamma_{\alpha\beta}f_\gamma$. Since $\nabla$ is compatible with
the metric, we have $a^\gamma_{\alpha\beta}
=-a^\beta_{\alpha\gamma}$. Thus we get
\begin{multline*}
\frac{1}{2}\sum_{\alpha,\beta}[c(f_\alpha)c(\nabla_{f_\alpha}
f_\beta)\nabla^{\cE}_{f_\beta}+c(f_\beta)c(\nabla_{f_\beta}
f_\alpha)\nabla^{\cE}_{f_\alpha}]\\
\begin{aligned}
=& \frac{1}{2}\sum_{\alpha,\beta,\gamma} [a^\gamma_{\alpha\beta}
c(f_\alpha)c(f_\gamma)\nabla^{\cE}_{f_\beta}+a^\gamma_{\alpha\beta}
c(f_\beta)c(f_\gamma)\nabla^{\cE}_{f_\alpha}]\\
=& - \frac{1}{2}\sum_{\alpha,\beta,\gamma} [a^\beta_{\alpha\gamma}
c(f_\alpha)c(f_\gamma)\nabla^{\cE}_{f_\beta}+a^\alpha_{\beta\gamma}
c(f_\beta)c(f_\gamma)\nabla^{\cE}_{f_\alpha}]
\\ =& - \frac{1}{2}\sum_{\alpha,\gamma} [
c(f_\alpha)c(f_\gamma)\nabla^{\cE}_{\nabla_{f_\alpha} f_\gamma}+
\sum_{\beta,\gamma}
c(f_\beta)c(f_\gamma)\nabla^{\cE}_{\nabla_{f_\beta} f_\gamma}] \\ =&
- \sum_{\alpha,\beta}
c(f_\alpha)c(f_\beta)\nabla^{\cE}_{\nabla_{f_\alpha} f_\beta}.
\end{aligned}
\end{multline*}
From the last three identities, we get
\begin{multline*}
(D^\prime_\cE)^2 =-\sum_{\alpha} (\nabla^{\cE}_{f_\alpha})^2 +
\frac{1}{2}\sum_{\alpha,\beta} c(f_\beta)c(f_\alpha)
R^{\cE}(f_\beta, f_\alpha)\\ + \frac{1}{2}\sum_{\alpha,\beta}
c(f_\beta)c(f_\alpha) \nabla^{\cE}_{[f_\beta, f_\alpha]} -
\sum_{\alpha,\beta}
c(f_\alpha)c(f_\beta)\nabla^{\cE}_{\nabla_{f_\alpha} f_\beta}.
\end{multline*}
Consider the last two terms in this identity. Using (\ref{e:R}),
we get
\begin{multline*}
\frac{1}{2}\sum_{\alpha,\beta} c(f_\beta)c(f_\alpha)
\nabla^{\cE}_{[f_\beta, f_\alpha]} - \sum_{\alpha,\beta}
c(f_\alpha)c(f_\beta)\nabla^{\cE}_{\nabla_{f_\alpha} f_\beta} \\
\begin{aligned}
= & \frac{1}{2}\sum_{\alpha,\beta} c(f_\beta)c(f_\alpha)
(\nabla^{\cE}_{\nabla_{f_\beta}f_\alpha}
-\nabla^{\cE}_{\nabla_{f_\alpha}f_\beta} -
\nabla^{\cE}_{\cR(f_\beta, f_\alpha)}) \\ & - \sum_{\alpha,\beta}
c(f_\alpha)c(f_\beta)\nabla^{\cE}_{\nabla_{f_\alpha} f_\beta}\\
= & - \frac{1}{2}\sum_{\alpha,\beta}
(c(f_\alpha)c(f_\beta)+c(f_\beta)c(f_\alpha))
\nabla^{\cE}_{\nabla_{f_\alpha}f_\beta}\\ & -
\frac{1}{2}\sum_{\alpha,\beta} c(f_\beta)c(f_\alpha)
\nabla^{\cE}_{\cR(f_\beta, f_\alpha)}\\ = &\sum_{\alpha}
\nabla^{\cE}_{\nabla_{f_\alpha}f_\alpha} -
\frac{1}{2}\sum_{\alpha,\beta} c(f_\beta)c(f_\alpha)
\nabla^{\cE}_{\cR(f_\beta, f_\alpha)}.
\end{aligned}
\end{multline*}
By (\ref{e:nabla}), we also get
\begin{equation}
\label{e:dd} \sum_{\alpha=1}^q (\nabla^{\cE}_{f_\alpha})^*
\nabla^{\cE}_{f_\alpha}
=-\sum_{\alpha=1}^q(\nabla^{\cE}_{f_\alpha})^2+\nabla^{\cE}_\tau+
\nabla^{\cE}_{\sum_\alpha\nabla_{f_\alpha}f_\alpha}.
\end{equation}
Thus, we arrive at the formula
\begin{equation}\label{e:sq1}
(D^\prime_\cE)^2=\sum_{\alpha=1}^q (\nabla^{\cE}_{f_\alpha})^*
\nabla^{\cE}_{f_\alpha} - \nabla^{\cE}_\tau
+\frac{1}{2}\sum_{\alpha,\beta}
c(f_\alpha)c(f_\beta)[R^{\cE}(f_\alpha,f_\beta)-\nabla^{\cE}_{\cR(f_\alpha,f_\beta)}].
\end{equation}

Now we turn to $(D_\cE)^2$:
\begin{multline*}
(D_\cE)^2 = (D^\prime_\cE)^2 - \frac12 [\sum_{\alpha=1}^q
(c(f_\alpha)c(\tau)+ c(\tau)c(f_\alpha))\nabla^{\cE}_{f_\alpha}\\ -
\frac12 \sum_{\alpha=1}^q c(f_\alpha)c(\nabla_{f_\alpha}\tau)
-\frac14\|\tau\|^2
\\ = (D^\prime_\cE)^2 + \nabla^{\cE}_{\tau} - \frac12
\sum_{\alpha=1}^q c(f_\alpha)c(\nabla_{f_\alpha}\tau)
-\frac14\|\tau\|^2.
\end{multline*}
Taking into account (\ref{e:sq1}), we get
\begin{equation}\label{e:sq2}
\begin{split}
 (D_\cE)^2 = & \sum_{\alpha=1}^q (\nabla^\cE_{f_\alpha})^*
\nabla^\cE_{f_\alpha} - \frac12 \sum_{\alpha=1}^q
c(f_\alpha)c(\nabla_{f_\alpha}\tau) -\frac14\|\tau\|^2
\\ & +\frac{1}{2}\sum_{\alpha,\beta}
c(f_\alpha)c(f_\beta)[R^{\cE}(f_\alpha,f_\beta)-\nabla^\cE_{\cR(f_\alpha,f_\beta)}].
\end{split}
\end{equation}
Finally, we use the formula
\[
\frac{1}{2}\sum_{\alpha,\beta}
c(f_\alpha)c(f_\beta)R^{\cE}(f_\alpha,f_\beta)=\frac{K}{4}
+\frac{1}{2}\sum_{\alpha,\beta}
c(f_\alpha)c(f_\beta)R^{\cE/S}(f_\alpha,f_\beta),
\]
that completes the proof of Theorem~\ref{t:Lich}.
\medskip\par
There is a natural action of $Cl(Q_x)$ on $\Lambda Q_x$ given by the
formula
\begin{equation}
\label{e:cliff} c(f)=\varepsilon_{f^*}-i_{f}, \quad f\in Q_x,
\end{equation}
where $\varepsilon_{f^*}$ denotes the exterior product by the
covector $f^*\in Q_x^*$ dual to $f$, $i_f$ the interior product by
$f$.

Recall that the symbol map $\sigma : Cl(Q_x) \to \Lambda Q_x$ is
defined by
\[
\sigma (a)=c(a)1, \quad a\in Cl(Q_x)
\]
and the quantization map ${\mathbf c} = \sigma^{-1} : \Lambda Q_x
\to Cl(Q_x)$ is given by
\[
{\mathbf c}(f_{i_1}\wedge f_{i_2}\wedge \ldots f_{i_k})=
c(f_{i_1})c(f_{i_2}) \ldots c(f_{i_k}),
\]
where $\{f_1, f_2, \ldots , f_q\}$ is an orthonormal base in $Q_x$.
These maps satisfy
\begin{equation}\label{e:sigma}
\sigma({\mathbf c}(v){\mathbf c}(\omega))=c(v)\omega,\quad v\in Q_x,
\quad \omega \in \Lambda(Q_x).
\end{equation}
So, for any $\omega_1 \in \Lambda^iQ_x$ and $\omega_2 \in
\Lambda^jQ_x$ we have
\[
\sigma({\mathbf c}(\omega_1){\mathbf c}(\omega_2))= \omega_1\wedge
\omega_2 \mod \Lambda^{i+j-2}Q_x.
\]

By (\ref{e:cliff}) and (\ref{e:sigma}), we have
\[
\sigma(\sum_{\alpha=1}^q c(f_\alpha)c(\nabla_{f_\alpha}\tau))
=\sum_{\alpha=1}^q c(f_\alpha)\nabla_{f_\alpha}\tau
=\sum_{\alpha=1}^q
(\varepsilon_{f^*_\alpha}-i_{f_\alpha})\nabla_{f_\alpha}\tau.
\]
Recall the following lemma.
\begin{lem}\label{l:dH}
Let $f_1, f_2, \ldots, f_q$ be a local orthonormal basis of $T^HM$
and $f^*_1, f^*_2, \ldots, f^*_q$ be the dual basis of $T^HM^*$.
Then on $C^\infty(M,\Lambda T^HM^{*})$ we have
\[
d_H=\sum_{\alpha=1}^q\varepsilon_{f^*_\alpha}
\nabla_{f_\alpha},\quad d^*_H=-\sum_{\alpha=1}^q i_{f_\alpha}
\nabla_{f_\alpha} + i_\tau.
\]
\end{lem}
By Lemma~\ref{l:dH}, we have
\[
\sigma(\sum_{\alpha=1}^q c(f_\alpha)c(\nabla_{f_\alpha}\tau)) =
d_H\tau +d^*_H \tau - \|\tau\|^2.
\]

Assume that the bundle-like metric $g_M$ on $M$ satisfies the
assumption: the mean curvature form $\tau$ is a basic one-form. As
shown by Dominguez \cite{Dominguez}, such a bundle-like metric
exists for any Riemannian foliation. Under this assumption, we
have \cite{Kamber-To} (see also \cite{Tondeur-book}): $d\tau=0$.
This fact implies that
\[
\sigma(\sum_{\alpha=1}^q c(f_\alpha)c(\nabla_{f_\alpha}\tau)) =
d^*_H \tau - \|\tau\|^2 \in C^\infty(M,\Lambda^0T^*M)=C^\infty(M),
\]
and, therefore,
\[
\sum_{\alpha=1}^q c(f_\alpha)c(\nabla_{f_\alpha}\tau) = d^*_H \tau
- \|\tau\|^2.
\]
So we come to the following consequence (cf. \cite{G-K91}):

\begin{thm}
Assume that the bundle-like metric $g_M$ on $M$ satisfies the
assumption: the mean curvature form $\tau$ is a basic one-form.
Then we have:
\begin{multline*}
(D_\cE)^2 = \sum_{\alpha=1}^q (\nabla^\cE_{f_\alpha})^*
\nabla^\cE_{f_\alpha} - \frac12 d^*_H \tau + \frac14\|\tau\|^2
\\  + \frac{K}{4} +\frac{1}{2}\sum_{\alpha,\beta}
c(f_\alpha)c(f_\beta)[R^{\cE/S}(f_\alpha,f_\beta)-\nabla_{\cR(f_\alpha,f_\beta)}],
\end{multline*}
where $f_1,\ldots, f_q$ is a local orthonormal basis of $T^HM$.
\end{thm}

\section{Transversal Bochner formula}\label{s:Bochner}
In this Section, we derive the Lichnerowicz formula for the
transverse Laplacian on a compact manifold $M$ equipped with a
Riemannian foliation $\cF$, which can be naturally called a Bochner
formula.

\subsection{The transverse signature and Laplace operators}
Suppose that $(M,{\mathcal F})$ is a compact Riemannian foliated
manifold equipped with a bundle-like metric $g_M$.  The
decomposition (\ref{e:decomp}) induces a bigrading on $\Lambda
T^{*}M$: $$ \Lambda^k
T^{*}M=\bigoplus_{i=0}^{k}\Lambda^{i,k-i}T^{*}M, $$ where
\[
\Lambda^{i,j}T^{*}M=\Lambda^{i}T\cF^{*}\otimes \Lambda^{j} T^HM^{*}.
\]

In this bigrading, the de Rham differential $d$ can be written as
$$ d=d_F+d_H+\theta, $$ where $d_F$ and $d_H$ are first order
differential operators (the tangential de Rham differential and the
transversal de Rham differential accordingly), and $\theta$ is a
zero order differential operator.

The transverse signature operator is a first order differential
operator in $C^{\infty}(M,\Lambda T^HM^{*})$ given by
\[
D_H=d_H + d^*_H,
\]
and the transversal Laplacian is a second order transversally
elliptic differential operator in $C^{\infty}(M,\Lambda T^HM^{*})$
given by
\[
\Delta_H=d_Hd^*_H+d^*_Hd_H.
\]

\begin{thm}\label{t:Bochner}
Let $f_1,\ldots, f_q$ be a local orthonormal basis of $T^HM$. Then
we have the following formula
\[
\Delta_H=\sum_{\alpha=1}^q \nabla^*_{f_\alpha} \nabla_{f_\alpha}+
\sum_{\alpha=1}^q \varepsilon_{f^*_\alpha}
i_{\nabla_{f_\alpha}\tau} - \sum_{\alpha,\beta}
\varepsilon_{f_\alpha} i_{f_\beta}\left(R(f_\alpha, f_\beta) -
\nabla_{\cR(f_\alpha,f_\beta)}\right).
\]
\end{thm}

We give two proofs of Theorem~\ref{t:Bochner}. The first proof
derives the theorem from Theorem~\ref{t:Lich}, whereas the second
proof is direct and makes no use of Theorem~\ref{t:Lich}.

\subsection{The first proof}
Consider a transverse Clifford module $\cE=\Lambda T^HM^*$ which
equipped with a natural leafwise flat Clifford connection and the
corresponding transverse Dirac operator $D_{\Lambda T^HM^*}$ acting
in $C^\infty(M,\Lambda T^HM^*)$. The Clifford action of $Cl(Q)$ on
$\cE$ is defined by the formula (\ref{e:cliff}). By Lemma~\ref{l:dH}
and (\ref{e:cliff}), we have
\begin{equation}\label{e:sign}
 D_{\Lambda T^HM^*} =
d_H+d^*_H-\frac12(\varepsilon_{\tau^*}+i_{\tau}).
\end{equation}
By (\ref{e:sign}), it follows that
\begin{align*}
\Delta_H= & \left(D_{\Lambda
T^HM^*}+\frac12(\varepsilon_{\tau^*}+i_{\tau})\right)^2 -
d^2_H-(d^*_H)^2 \\ = & D_{\Lambda T^HM^*}^2- d^2_H-(d^*_H)^2
+\frac12\Big(D_{\Lambda T^HM^*}(\varepsilon_{\tau^*}+i_{\tau})
+(\varepsilon_{\tau^*}+i_{\tau})D_{\Lambda T^HM^*}\Big)\\ & +\frac14
(\varepsilon_{\tau^*}i_{\tau}+i_{\tau}\varepsilon_{\tau^*}).
\end{align*}
By Theorem~\ref{t:Lich}, it follows that
\begin{multline*}
D_{\Lambda T^HM^*}^2 = \sum_{\alpha=1}^q \nabla^*_{f_\alpha}
\nabla_{f_\alpha} - \frac12 \sum_{\alpha=1}^q
(\varepsilon_{f_\alpha} - i_{f_\alpha})
(\varepsilon_{\nabla_{f_\alpha}\tau} - i_{\nabla_{f_\alpha}\tau})
-\frac14\|\tau\|^2
\\  +\frac{1}{2}\sum_{\alpha,\beta}
(\varepsilon_{f_\alpha} - i_{f_\alpha}) (\varepsilon_{f_\beta} -
i_{f_\beta}) [R^{\Lambda
T^HM^{*}}(f_\alpha,f_\beta)-\nabla_{\cR(f_\alpha,f_\beta)}].
\end{multline*}
As in the classical case, we have
\begin{gather}\label{e:eR}
\frac{1}{2}\sum_{\alpha,\beta} \varepsilon_{f_\alpha}
\varepsilon_{f_\beta} R^{\Lambda T^HM^{*}}(f_\alpha,f_\beta) =
0,\\ \label{e:iR} \frac{1}{2}\sum_{\alpha,\beta} i_{f_\alpha}
i_{f_\beta} R^{\Lambda T^HM^{*}}(f_\alpha,f_\beta) = 0.
\end{gather}

The following lemma seems to be well known, but we didn't find an
appropriate reference.

\begin{lemma}\label{l:sq}
We have
\[
d^2_H=-\frac{1}{2}\sum_{\alpha,\beta}
\varepsilon_{f_\alpha}\varepsilon_{f_\beta}\nabla_{\cR(f_\alpha,f_\beta)}
\]
and
\[
(d^*_H)^2=-\frac{1}{2}\sum_{\alpha,\beta}
i_{f_\alpha}i_{f_\beta}\nabla_{\cR(f_\alpha,f_\beta)}-\sum_\alpha
i_{f_\alpha} i_{\nabla_{f_\alpha}\tau}.
\]
\end{lemma}

\begin{proof}
(1) By Lemma~\ref{l:dH}, we have
\[
d^2_H\omega = \sum_{\alpha, \beta} f_\alpha \wedge
\nabla_{f_\alpha}f_\beta\wedge \nabla_{f_\beta}\omega +
\sum_{\alpha, \beta} f_\alpha \wedge f_\beta\wedge
\nabla_{f_\alpha} \nabla_{f_\beta}\omega.
\]
As above, write $\nabla_{f_\alpha}f_\beta=\sum_\gamma
a^\gamma_{\alpha\beta}f_\gamma$, where $a^\gamma_{\alpha\beta}
=-a^\beta_{\alpha\gamma}$. Then, for the first term, we have
\begin{align*}
\sum_{\alpha, \beta} f_\alpha \wedge
\nabla_{f_\alpha}f_\beta\wedge \nabla_{f_\beta}\omega  & =
\sum_{\alpha, \beta,\gamma} a^\gamma_{\alpha\beta} f_\alpha \wedge
f_\gamma\wedge \nabla_{f_\beta}\omega \\ & = - \sum_{\alpha,
\beta,\gamma} a^\beta_{\alpha\gamma} f_\alpha \wedge
f_\gamma\wedge \nabla_{f_\beta}\omega
\\ & = - \sum_{\alpha,\gamma} f_\alpha \wedge f_\gamma\wedge
\nabla_{\nabla_{f_\alpha}f_\gamma}\omega
\\ & = - \frac12 \sum_{\alpha,\gamma} f_\alpha \wedge f_\gamma\wedge
(\nabla_{\nabla_{f_\alpha}f_\gamma}-\nabla_{\nabla_{f_\gamma}f_\alpha})\omega
\\ & = - \frac12 \sum_{\alpha,\gamma} f_\alpha \wedge f_\gamma\wedge
(\nabla_{[f_\alpha,f_\gamma]+\cR(f_\alpha,f_\gamma)})\omega
\end{align*}
For the second term, we use the definition of the curvature $R$
and (\ref{e:eR}):
\begin{align*}
\sum_{\alpha, \beta} f_\alpha \wedge f_\beta\wedge
\nabla_{f_\alpha} \nabla_{f_\beta}\omega & =\frac12 \sum_{\alpha,
\beta} f_\alpha \wedge f_\beta\wedge (\nabla_{f_\alpha}
\nabla_{f_\beta}-\nabla_{f_\beta} \nabla_{f_\alpha})\omega \\ &
=\frac12 \sum_{\alpha, \beta} f_\alpha \wedge f_\beta\wedge
(\nabla_{[f_\alpha, f_\beta]}+R(f_\alpha,f_\beta))\omega
\\ & = \frac12 \sum_{\alpha, \beta} f_\alpha \wedge f_\beta\wedge
\nabla_{[f_\alpha, f_\beta]}\omega.
\end{align*}

(2) Similarly, using Lemma~\ref{l:dH}, we get
\[
(d^*_H)^2=\sum_{\alpha,\beta}^q i_{f_\alpha}
i_{\nabla_{f_\alpha}f_\beta} \nabla_{f_\beta} +
\sum_{\alpha,\beta}^q i_{f_\alpha} i_{f_\beta} \nabla_{f_\alpha}
\nabla_{f_\beta} - \sum_\alpha (i_\tau i_{f_\alpha}
\nabla_{f_\alpha}+i_{f_\alpha} \nabla_{f_\alpha}i_\tau).
\]
Repeating the same arguments as above, we obtain
\[
\sum_{\alpha,\beta}^q i_{f_\alpha} i_{\nabla_{f_\alpha}f_\beta}
\nabla_{f_\beta} + \sum_{\alpha,\beta}^q i_{f_\alpha} i_{f_\beta}
\nabla_{f_\alpha} \nabla_{f_\beta}=-\frac{1}{2}\sum_{\alpha,\beta}
i_{f_\alpha}i_{f_\beta}\nabla_{\cR(f_\alpha,f_\beta)}.
\]
For the third term, we have
\begin{align*}
\sum_\alpha (i_\tau i_{f_\alpha} \nabla_{f_\alpha}+i_{f_\alpha}
\nabla_{f_\alpha}i_\tau)& =\sum_\alpha \left((i_\tau i_{f_\alpha}
+i_{f_\alpha}i_\tau) \nabla_{f_\alpha}+ i_{f_\alpha}
i_{\nabla_{f_\alpha}\tau}\right)\\ & = \sum_\alpha i_{f_\alpha}
i_{\nabla_{f_\alpha}\tau}.
\end{align*}
\end{proof}

By (\ref{e:sq2}) and Lemma~\ref{l:sq}, it follows that
\begin{multline*}
D_{\Lambda T^HM^*}^2- d^2_H-(d^*_H)^2 = \sum_{\alpha=1}^q
\nabla^*_{f_\alpha} \nabla_{f_\alpha} - \frac12 \sum_{\alpha=1}^q
(\varepsilon_{f_\alpha} - i_{f_\alpha})
(\varepsilon_{\nabla_{f_\alpha}\tau} - i_{\nabla_{f_\alpha}\tau})
-\frac14\|\tau\|^2
\\ -\sum_{\alpha,\beta}
\varepsilon_{f_\alpha}i_{f_\beta} [R^{\Lambda
T^HM^{*}}(f_\alpha,f_\beta)-\nabla_{\cR(f_\alpha,f_\beta)}]+\sum_\alpha
i_{f_\alpha} i_{\nabla_{f_\alpha}\tau}.
\end{multline*}
Recall that $D_{\Lambda T^HM^*}$ is given by the formula
\[
D_{\Lambda T^HM^*}= \sum_{\alpha=1}^q (\varepsilon_{f_\alpha} -
i_{f_\alpha}) \left(\nabla^{\Lambda T^HM^{*}}_{f_\alpha}-\frac12
g_M(\tau, f_\alpha) \right)
\]
Using the identity
$(\varepsilon_{u}-i_{u})(\varepsilon_{v}+i_{v})+(\varepsilon_{v}+i_{v})(\varepsilon_{u}-i_{u})=0$
for any $u$ and $v$, we get
\begin{multline*}
D_{\Lambda T^HM^*}(\varepsilon_{\tau^*}+i_{\tau})
+(\varepsilon_{\tau^*}+i_{\tau})D_{\Lambda T^HM^*}\\
\begin{aligned}
& = \sum_{\alpha=1}^q (\varepsilon_{f_\alpha} - i_{f_\alpha})
\left[\nabla^{\Lambda T^HM^{*}}_{f_\alpha}-\frac12 g_M(\tau,
f_\alpha), \varepsilon_{\tau^*}+i_{\tau} \right] \\ & =
\sum_{\alpha=1}^q (\varepsilon_{f_\alpha} - i_{f_\alpha}) (
\varepsilon_{\nabla_{f_\alpha} \tau^*}+i_{\nabla_{f_\alpha}\tau}).
\end{aligned}
\end{multline*}
The above identities and the formula
$\varepsilon_{\tau^*}i_{\tau}+i_{\tau}\varepsilon_{\tau^*}=\|\tau\|^2$
immediately complete the proof.

\subsection{A direct proof}
Here we indicate a direct proof of Theorem~\ref{t:Bochner}. By
Lemma~\ref{l:dH}, we have
\[
d^*_Hd_H = -\sum_{\alpha,\beta} i_{f_\alpha}
\varepsilon_{\nabla_{f_\alpha} f^*_\beta} \nabla_{f_\beta}
-\sum_{\alpha,\beta} i_{f_\alpha} \varepsilon_{f^*_\beta}
\nabla_{f_\alpha} \nabla_{f_\beta} + \sum_{\alpha=1}^q i_\tau
\varepsilon_{f^*_\alpha} \nabla_{f_\alpha}
\]
and
\[
d_Hd^*_H = -\sum_{\alpha,\beta}
\varepsilon_{f_\alpha}i_{\nabla_{f_\alpha}f_\beta} \nabla_{f_\beta}
-\sum_{\alpha,\beta} \varepsilon_{f_\alpha} i_{f_\beta}
\nabla_{f_\alpha} \nabla_{f_\beta} + \sum_{\alpha=1}^q
\varepsilon_{f^*_\alpha} \nabla_{f_\alpha}i_\tau.
\]
From these identities, it follows that
\begin{align*}
\Delta_H = & -\sum_{\alpha,\beta} (i_{f_\alpha}
\varepsilon_{\nabla_{f_\alpha} f^*_\beta} +
\varepsilon_{f_\alpha}i_{\nabla_{f_\alpha}f_\beta})
\nabla_{f_\beta} -\sum_{\alpha,\beta} (i_{f_\alpha}
\varepsilon_{f^*_\beta}+\varepsilon_{f_\alpha} i_{f_\beta})
\nabla_{f_\alpha} \nabla_{f_\beta}\\ & + \sum_{\alpha=1}^q (i_\tau
\varepsilon_{f^*_\alpha} \nabla_{f_\alpha}
+\varepsilon_{f^*_\alpha} \nabla_{f_\alpha}i_\tau).
\end{align*}
Writing $\nabla_{f_\alpha}f_\beta=\sum_\gamma
a^\gamma_{\alpha\beta}f_\gamma$, where $a^\gamma_{\alpha\beta}
=-a^\beta_{\alpha\gamma}$, we get
\begin{align*}
\sum_{\alpha,\beta} (i_{f_\alpha} \varepsilon_{\nabla_{f_\alpha}
f^*_\beta} + \varepsilon_{f_\alpha}i_{\nabla_{f_\alpha}f_\beta})
\nabla_{f_\beta} = & \sum_{\alpha,\beta,\gamma}
a^\gamma_{\alpha\beta} (i_{f_\alpha} \varepsilon_{f_\gamma} +
\varepsilon_{f_\alpha}i_{f_\gamma}) \nabla_{f_\beta}\\  = &
-\sum_{\alpha,\beta,\gamma} a^\beta_{\alpha\gamma} (i_{f_\alpha}
\varepsilon_{f_\gamma} + \varepsilon_{f_\alpha}i_{f_\gamma})
\nabla_{f_\beta}\\  = & -\sum_{\alpha,\gamma} (i_{f_\alpha}
\varepsilon_{f_\gamma} + \varepsilon_{f_\alpha}i_{f_\gamma})
\nabla_{\nabla_{f_\alpha}f_\gamma} \\  = & -\sum_{\alpha,\gamma}
(i_{f_\alpha} \varepsilon_{f_\gamma} +
\varepsilon_{f_\gamma}i_{f_\alpha})
\nabla_{\nabla_{f_\alpha}f_\gamma}\\ & -\sum_{\alpha,\gamma} (
\varepsilon_{f_\alpha}i_{f_\gamma}-
\varepsilon_{f_\gamma}i_{f_\alpha})
\nabla_{\nabla_{f_\alpha}f_\gamma}\\  = & - \nabla_{\sum_{\alpha}
\nabla_{f_\alpha}f_\alpha} -\sum_{\alpha,\gamma}
\varepsilon_{f_\alpha}i_{f_\gamma}
\nabla_{\nabla_{f_\alpha}f_\gamma-\nabla_{f_\gamma}f_\alpha} \\  =
& - \nabla_{\sum_{\alpha} \nabla_{f_\alpha}f_\alpha}
-\sum_{\alpha,\gamma} \varepsilon_{f_\alpha}i_{f_\gamma}
\nabla_{[f_\alpha,f_\gamma]+\cR(f_\alpha,f_\gamma)}.
\end{align*}
We also have
\begin{align*}
\sum_{\alpha,\beta} (i_{f_\alpha}
\varepsilon_{f^*_\beta}+\varepsilon_{f_\alpha} i_{f_\beta})
\nabla_{f_\alpha} \nabla_{f_\beta} = & \sum_{\alpha,\beta}
(i_{f_\alpha} \varepsilon_{f^*_\beta}+\varepsilon_{f_\beta}
i_{f_\alpha}) \nabla_{f_\alpha} \nabla_{f_\beta}\\ & +
\sum_{\alpha,\beta} (\varepsilon_{f_\alpha}
i_{f_\beta}-\varepsilon_{f_\beta} i_{f_\alpha}) \nabla_{f_\alpha}
\nabla_{f_\beta}\\ = & \sum_{\alpha} (\nabla_{f_\alpha})^2 +
\sum_{\alpha,\beta} \varepsilon_{f_\alpha}
i_{f_\beta}(\nabla_{f_\alpha} \nabla_{f_\beta}- \nabla_{f_\beta}
\nabla_{f_\alpha})\\ = & \sum_{\alpha} (\nabla_{f_\alpha})^2 +
\sum_{\alpha,\beta} \varepsilon_{f_\alpha}
i_{f_\beta}(\nabla_{[f_\alpha,f_\beta]}+R(f_\alpha, f_\beta)).
\end{align*}
Taking into account (\ref{e:dd}), we immediately complete the proof
of Theorem~\ref{t:Bochner}.

\end{document}